\documentclass[reqno]{amsart}
\usepackage{amssymb}
\usepackage{graphicx}

\usepackage[usenames, dvipsnames]{color}
\usepackage{verbatim}
\usepackage{mathrsfs}
\usepackage{bm}
\usepackage{cite}
\usepackage{mathtools}
\usepackage{scalerel}

\numberwithin{equation}{section}

\newtheorem{theorem}{Theorem}[section]
\newtheorem{corollary}[theorem]{Corollary}
\newtheorem{lemma}[theorem]{Lemma}
\newtheorem{prop}[theorem]{Proposition}

\theoremstyle{definition}
\newtheorem{remark}[theorem]{Remark}

\theoremstyle{definition}
\newtheorem{definition}[theorem]{Definition}

\theoremstyle{definition}

\theoremstyle{definition}

\makeatletter
\def\dashint{\operatorname%
{\,\,\text{\bf-}\kern-.98em\DOTSI\intop\ilimits@\!\!}}
\makeatother

\def\\det{\text{\det}}

\def\Xint#1{\mathchoice
 {\XXint\displaystyle\textstyle{#1}}%
 {\XXint\textstyle\scriptstyle{#1}}%
 {\XXint\scriptstyle\scriptscriptstyle{#1}}%
 {\XXint\scriptscriptstyle\scriptscriptstyle{#1}}%
 \!\int}
\def\XXint#1#2#3{{\setbox0=\hbox{$#1{#2#3}{\int}$}
  \vcenter{\hbox{$#2#3$}}\kern-.5\wd0}}

\def\dashint{\Xint-}

\def\.5{\frac{1}{2}}

\newcommand{\RN}[1]{%
  \textup{\uppercase\expandafter{\romannumeral#1}}%
}

\renewcommand{\epsilon}{\varepsilon}

\newcounter{marnote}

\begin{document}

\title[The ideal MHD equations as a differential inclusion]{Nonuniqueness for high-dimensional ideal MHD equations via differential inclusion}

\author[C.X. Miao]{Changxing Miao}
\address[C.X. Miao] {Institute of Applied Physics and Computational Mathematics, P.O. Box 8009, Beijing, 100088, China.}
\email{miao\_changxing@iapcm.ac.cn}

\author[Z.W. Zhao]{Zhiwen Zhao}

\address[Z.W. Zhao]{School of Mathematics and Physics, University of Science and Technology Beijing, Beijing 100083, China.}
\email{zwzhao365@163.com}


\date{\today} 


\maketitle
\begin{abstract}
In this paper, we establish the non-uniqueness of solutions to the ideal magnetohydrodynamics equations in any dimension greater than three by proving the existence of infinitely many compactly supported weak solutions. In particular, these solutions fail to conserve the total energy. Our proof relies on the differential inclusion framework tailored to the geometry of ideal MHD system, which enables the simultaneous use of Baire category method and convex integration scheme.

\end{abstract}

\maketitle



\section{Introduction and main results}
This paper focuses on the study of the system of ideal magnetohydrodynamics (MHD), which is formed by linking the incompressible Euler equations and the Faraday–Maxwell system through Ohm’s law. While the viscous and resistive MHD equations are commonly employed in modeling conducting fluids such as plasmas and liquid metals (see e.g. \cite{B1993,D2001,GBL2006,ST1983}), the inviscid and non-resistive regime, referred to as ideal MHD, features an abundance of mathematical structures \cite{AK1998}. The idea MHD equations are formulated as follows:
\begin{align}\label{IMHD}
\begin{cases}	
\partial_{t}u+\mathrm{div}(u\otimes u-b\otimes b)+\nabla p=f,\\
\partial_{t}b+\mathrm{div}(b\otimes u-u\otimes b)=g,\\
\mathrm{div}u=\mathrm{div}b=0,
\end{cases}	\quad\text{in $\mathbb{R}_{x}^{n}\times\mathbb{R}_{t}$},
\end{align}	
where $(u,b):\mathbb{R}_{x}^{n}\times\mathbb{R}_{t}\rightarrow\mathbb{R}^{n}\times\mathbb{R}^{n}$ represents the pair of velocity and magnetic fields,  $p:\mathbb{R}_{x}^{n}\times\mathbb{R}_{t}\rightarrow\mathbb{R}$ denotes the pressure, the vector fields $f$ and $g$ are prescribed external forcing terms. The local well-posedness theory for smooth solutions of the MHD system and the global existence theory for weak solutions essentially follow the same pattern as in the Euler and Navier–Stokes cases, see \cite{DL1972,ST1983}. In the absence of external forces ($f=g=0$), classical solutions to the ideal MHD system preserve the total energy, which is also the only known coercive conserved quantity (indeed, the Hamiltonian \cite{KPY2021,VD1978} of the system).\ Equivalently, $\int(|u(x,t)|^{2}+|b(x,t)|^{2})dx$ is invariant in time, which naturally identifies $L_{t}^{\infty}L_{x}^{2}$ as the energy space.\ This motivates the following definition of weak solutions to \eqref{IMHD}.
\begin{definition}
A pair of vector-valued functions $(u,b)\in L_{\mathrm{loc}}^{2}$ is termed as a weak solution of the ideal MHD equations in \eqref{IMHD} provided that for any $t\in\mathbb{R}$, the vector fields $u(\cdot,t)$ and $b(\cdot,t)$ are weakly divergence free, and
\begin{align*}
\begin{cases}	
\int\big(u\cdot\partial_{t}\varphi+(u\otimes u-b\otimes b):\nabla\varphi+f\cdot \varphi\big)dxdt=0,\\
\int\big(b\cdot\partial_{t}\varphi+(b\otimes u-u\otimes b):\nabla\varphi+g\cdot \varphi\big)dxdt=0,
\end{cases}
\end{align*}	 
for any test function $\varphi\in C_{0}^{\infty}(\mathbb{R}_{x}^{n}\times\mathbb{R}_{t};\mathbb{R}^{n})$ with $\mathrm{div}\varphi=0.$

\end{definition}	
The aim of this paper is to investigate the phenomenon of non-uniqueness for weak solutions that do not conserve the total energy. This study continues the line of research initiated in \cite{DS2009}, broadening the scope from the Euler equations to the full ideal MHD system in all dimensions higher than three. In \cite{DS2009} De Lellis and Sz\'{e}kelyhidi introduced a new perspective by reformulating the Euler system as a differential inclusion, which allows them to make use of Baire category argument and convex integration scheme to prove the non-uniqueness of non-conservative weak solutions. Althought earlier work \cite{S1993,S1997,S2000} had already revealed the existence of pathological weak solutions, the convex integration method developed in \cite{DS2009} has exhibited great robustness and undergone further adaptations and developments, leading to its successful application to a broad class of equations in fluid dynamics and their various variants. One line of research, stimulated by Onsager’s conjecture \cite{O1949} for the Euler equations, addressed the problem of improving the regularity of weak solutions up to the Onsager-critical level.\ This programme originated in \cite{DS2013}, was subsequently developed in \cite{BDIS2015,DS2014,DS201701}, and ultimately culminated in Isett’s work \cite{I2018}, see also \cite{BDSV2019}. For a more detailed introduction to the development of convex integration scheme, we refer to \cite{BV201902,BV2021,DS2017,BMNV2023} and the references therein.

With regard to the ideal MHD equations in three dimensions, Bronzi, Lopes Filho and Nussenzveig Lopes \cite{BLN2015} constructed bounded wild solutions in a symmetry-reduced form of $u(x_{1},x_{2},x_{3},t)=(u_{1}(x_{1},x_{2},t),u_{2}(x_{1},x_{2},t),0)$ and $b(x_{1},x_{2},x_{3},t)=(0,0,b(x_{1},x_{2},t))$ by adapting the proof given in \cite{DS2009} to the two-dimensional Euler equations coupled with a passive tracer, which can be regarded as the symmetry-reduced three-dimensional ideal MHD system.\ This simplification essentially removes the intricate coupling effects of the magnetic field within the flow, thereby making the other two conservation laws including the cross helicity and the magnetic helicity all vanish identically.\ Motivated by the need to analyze the non-conservation of weak solutions with respect to these two quantities and using the intermittent convex integration scheme developed in \cite{BV2019},  Beekie, Buckmaster and Vicol \cite{BBV2020} constructed $L_{t}^{\infty}L^{2}_{x}$ weak solutions to the ideal MHD equations whose magnetic helicity is not conserved over time, which also provided an example of weak solutions to the ideal MHD system with finite energy that do not arise from weak ideal limits of Leray-Hopf solutions to the MHD equations according to Taylor’s conjecture \cite{FL2020}. Faraco, Lindberg and Sz\'{e}kelyhidi \cite{FLS2021} proved that the three-dimensional ideal magnetohydrodynamic equations admit infinitely many bounded compactly supported weak solutions, which violate the conservation of total energy and cross helicity but retain zero magnetic helicity.\ The proof in \cite{FLS2021} was based on the idea in \cite{DS2009} and the ideal MHD equations were relaxed into the linear conservation laws corresponding to the velocity and magnetic fields. While the relaxation for the velocity field involves a divergence-type structure, the relaxation for the magnetic field is provided by the Faraday–Maxwell equations with a curl-type structure.\ Within this relaxed formulation, one can readily recognize the conservation of magnetic helicity as an instance of compensated compactness in accordance with Tartar’s work \cite{T1983} when applied to the Maxwell system. Subsequently, they \cite{FLS2024} further constructed bounded turbulent solutions of the three-dimensional ideal MHD which dissipate energy and cross helicity but preserve magnetic helicity at any given constant value.\ Moreover, they \cite{FLS2024} also demonstrated that the $L^{3}$ integrability condition for magnetic helicity conservation in \cite{KL2007} is sharp.

The main results of this paper are now listed as follows.
\begin{theorem}\label{Main01}
For any $n\geq4$, there exist nontrivial compactly supported weak solutions $(u,b,p)\in L^{\infty}(\mathbb{R}^{n}_{x}\times\mathbb{R}_{t};\mathbb{R}^{n}\times \mathbb{R}^{n}\times\mathbb{R})$ of the ideal MHD in \eqref{IMHD} with $f=g=0$.
\end{theorem}	
\begin{remark}
For any given $c_{0}>0$, applying the proof of Theorem \ref{Main01} with minor modification, we can find a solution $(u,b,p)$ as in Theorem \ref{Main01} such that
\begin{align*}
\begin{cases}	
\int\big(|u(x,t)|^{2}+|b(x,t)|^{2}\big)dx=c_{0},&\text{for a.e. }t\in[-1,1],\\
u(x,t)=b(x,t)=0,&\text{for $|t|>1$}.
\end{cases}	
\end{align*}	
Clearly, these solutions break the law of total energy conservation.
\end{remark}

\begin{theorem}\label{Main02}
For every solution $(u,b,p)$ obtained in Theorem \ref{Main01}, there exist a sequence of solutions $\{(u_{k},b_{k},p_{k},f_{k},g_{k})\}\subset C_{0}^{\infty}$ to the ideal MHD in \eqref{IMHD} such that $\|(u_{k},b_{k},p_{k})\|_{L^{\infty}}$ remains uniformly bounded, and for any $m\geq 2$,
\begin{align*}
(u_{k},b_{k},p_{k})\longrightarrow(u,b,p)\quad\text{in}\;L^{m},\quad(f_{k},g_{k})\longrightarrow(0,0)\quad\text{in $W^{-1,m}$}.
\end{align*}	

\end{theorem}

\begin{remark}
Theorem \ref{Main02} reveals that the nonconservative weak solutions found in Theorem \ref{Main01} can be realized in the strong limit as the forces vanish for a sequence of smooth solutions of the forced ideal MHD equations.

\end{remark}

The rest of this article is structured as follows.\ In Section \ref{SEC02}, we recast the ideal MHD equations in the form of a differential inclusion and construct highly oscillatory localized plane-wave solutions to the linearized relaxation system in every direction inside the wave cone.\ Section \ref{SEC03} begins by convexifying the nonlinear pointwise constraint in the differential inclusion framework. Then for any relaxed solution lying in the interior of this convex set, we add the highly oscillatory building blocks constructed in Section \ref{SEC02} in order to derive an energy growth inequality, which is the crucial step in iteratively producing a sequence of smooth relaxed solutions approximating a weak solution of the ideal MHD equations.\ In Section \ref{SEC04}, we complete the proof of Theorem \ref{Main01} by presenting two distinct approaches: the Baire category method and the convex integration scheme. Theorem \ref{Main02} then follows as a direct consequence.

\section{Relaxation of ideal MHD and localized plane wave}\label{SEC02}
Let $I_{n}$ be the $n\times n$ identity matrix. Write $\mathscr{S}^{n}$ for the linear space of symmetric $n\times n$ matrices, and $\mathscr{S}_{0}^{n}$ for its trace-free subspace. Let $\mathscr{S}_{-}^{n}$ represent the linear space of skew-symmetric $n\times n$ matrices. Then based on Tartar's framework \cite{T1979}, we see that the ideal MHD equations can be reformulated as follows:
\begin{align}\label{TR01}
\begin{cases}	
	\partial_{t}u+\mathrm{div}M+\nabla q=0,\\
	\partial_{t}b+\mathrm{div}Q=0,\\
	\mathrm{div}u=\mathrm{div}b=0,
\end{cases}	\quad\text{in $\mathbb{R}^{n}_{x}\times\mathbb{R}_{t}$},
\end{align}	
where $(u,b,M,Q,q)\in\mathbb{R}^{n}\times\mathbb{R}^{n}\times\mathscr{S}_{0}^{n}\times\mathscr{S}_{-}^{n}\times\mathbb{R}$ satisfies 
\begin{align}\label{TR02}
\begin{cases}
M=u\otimes u-b\otimes b-(|u|^{2}-|b|^{2})I_{n}/n,\\
Q=b\otimes u-u\otimes b,\\
q=p+(|u|^{2}-|b|^{2})/n.
\end{cases}
\end{align}	
Disregarding the nonlinear constraint in \eqref{TR02}, the equation in \eqref{TR01} is termed as a linearized relaxation system of idea MHD. For $(u,b,M,Q,q)\in\mathbb{R}^{n}\times\mathbb{R}^{n}\times\mathscr{S}_{0}^{n}\times\mathscr{S}_{-}^{n}\times\mathbb{R}$, let
\begin{align}\label{M01}
	W=\begin{pmatrix}
		M+qI_{n}&u\\
		u&0\\
		Q&b\\
		b&0	
	\end{pmatrix}.		
\end{align}
By introducing the new coordinate $y=(x,t)\in\mathbb{R}^{n+1}$, the relaxation equation in \eqref{TR01} can be rewritten as 
\begin{align}\label{TR03}
\mathrm{div}_{y}W=0,\quad\mathrm{in}\; \mathbb{R}_{x}^{n}\times\mathbb{R}_{t}.	
\end{align}	
Denote
\begin{align*}
\mathscr{K}_{1}=\{U\in\mathscr{S}^{n+1}|\,U_{(n+1)(n+1)}=0\},	
\end{align*}
and
\begin{align*}
\mathscr{K}_{2}=\left\{\begin{pmatrix}
	Q&b\\
	b&0	
\end{pmatrix}\bigg|	\,Q\in \mathscr{S}_{-}^{n},\,b\in\mathbb{R}^{n}\right\}.	
\end{align*}
Then we define		
\begin{align*}
\mathscr{M}=\left\{\begin{pmatrix} U\\ V\end{pmatrix}\bigg|\,U\in\mathscr{K}_{1},\,V\in\mathscr{K}_{2}\right\}.	
\end{align*}	
The relaxation equation in \eqref{TR03} admits a class of low-rank solutions of plane-wave type as follows: 
\begin{align*}
W(y)=h(y\cdot\xi)\overline{W},\quad\xi\in\mathbb{R}^{n+1}\setminus\{0\},	
\end{align*}	
where $h:\mathbb{R}\rightarrow\mathbb{R}$ is a smooth periodic function and 
\begin{align}\label{W01}
\overline{W}\in\{W\in\mathscr{M}\,|\,\exists \, \xi \in \mathbb{R}^{n+1} \setminus \{0\}\;\text{with }W\xi=0\}.	
\end{align}	 
The wave cone given by \eqref{W01} can be restated as
\begin{align*}
\Lambda:=\{(u,b,M,Q,q)\in\mathbb{R}^{n}\times\mathbb{R}^{n}\times\mathscr{S}_{0}^{n}\times\mathscr{S}_{-}^{n}\times\mathbb{R}\,|\,\exists \, \xi \in \mathbb{R}^{n+1} \setminus \{0\}\;\text{with }W\xi=0\},	
\end{align*}	
where $W$ is defined by \eqref{M01}. Remark that the correspondence
\begin{align*}
(u,b,M,Q,q)\in\mathbb{R}^{n}\times\mathbb{R}^{n}\times\mathscr{S}_{0}^{n}\times\mathscr{S}_{-}^{n}\times\mathbb{R}\longmapsto W\in\mathscr{M} 
\end{align*}
provides a linear isomorphism. Observe that since the dimension of $W$ is $n(n+2)$, it contains $n(n+2)$
independent degrees of freedom.\ Therefore, for any fixed $\xi\in\mathbb{R}^{n+1}\setminus\{0\}$, the dimension of the set $\{W\in\mathscr{M}\,|\,\exists\,W\in\mathscr{M}\text{ with }W\xi=0\}$ is $n^{2}-2$ and thus the wave cone $\Lambda$ is of considerable size. 

For later use, we now construct highly oscillatory localized plane waves as follows.
\begin{prop}\label{pro01}
Assume that $\overline{W}\in\{W\in\mathscr{M}\,|\,\exists \, \xi \in \mathbb{R}^{n+1} \setminus \{0\}\;\text{with }W\xi=0\}$ with $\overline{W}e_{n+1}\neq0$. Let $\sigma$ be the line segment joining $-\overline{W}$ and $\overline{W}$. Then for any $\delta,m>0$, one can construct a  solution $W\in C^{\infty}(\mathbb{R}^{n+1};\mathscr{M})$ of \eqref{TR03} satisfying that
\begin{align*}
\mathrm{supp}\,W\subset B_{1}(0),\;\,\sup\limits_{y\in B_{1}(0)}\mathrm{dist}(W(y),\sigma)<\delta,\;\,\dashint_{B_{1}(0)}|W(y)e_{n+1}|^{m}dy\geq\alpha|\overline{W}e_{n+1}|^{m},
\end{align*}
for some constant $\alpha=\alpha(n,m)>0$, where the notation $\dashint_{B_{1}(0)}$ denotes the integral average over the unit ball.

\end{prop}	

Before proving Proposition \ref{pro01}, we first list two useful lemmas. Let 
\begin{align*}
\mathscr{G}=\{A\in\mathbb{R}^{(n+1)\times(n+1)}|\,\det A\neq0,\,Ae_{n+1}=e_{n+1}\}.
\end{align*}
For any $U\in\mathscr{K}_{1}$, $V\in\mathscr{K}_{2}$ and $A\in\mathscr{G}$, denote
\begin{align*}
(U,V)^{t}=\begin{pmatrix}
	U\\V	
\end{pmatrix},\quad A^{t}(U,V)^{t}A=\begin{pmatrix}
A^{t}UA\\A^{t}VA	
\end{pmatrix}.
\end{align*}
\begin{lemma}\label{lem01}
 For each $A\in\mathscr{G}$ and any divergence-free matrix field $W\in\mathscr{M}$, the matrix field $T(y):=A^{t}W(A^{-t}y)A$ remains divergence-free.

\end{lemma}	

Set
\begin{align*}
\mathscr{H}_{1}=\big\{&E=(E_{ij}^{kl})\in(C^{\infty}(\mathbb{R}^{n+1}))^{(n+1)^{4}}|\,E_{ij}^{kl}=-E_{ji}^{kl},\,E_{ij}^{kl}=-E_{ij}^{lk},\\
&E_{(n+1)i}^{(n+1)j}=0,\,1\leq i,j\leq n+1\big\},\\
\mathscr{H}_{2}=\big\{&E=(E_{ij}^{kl})\in \mathscr{H}_{1}|\,E^{il}_{kj}=0,\,1\leq i,j\leq n,1\leq k,l\leq n+1\big\},\\
\mathscr{H}_{3}=\big\{&E=(E_{ij}^{kl})\in(C^{\infty}(\mathbb{R}^{n+1}))^{(n+1)^{4}}|\,E_{ij}^{kl}=-E_{ji}^{kl},\,E_{ij}^{kl}=-E_{ij}^{lk},\\
&E_{kj}^{(n+1)l}=E_{k(n+1)}^{jl}=0,\,1\leq j,k,l\leq n+1\big\}.
\end{align*}	
For any $E\in \mathscr{H}_{1},F\in \mathscr{H}_{2},G\in \mathscr{H}_{3}$, define the linear operator $\mathscr{L}$ as follows:
\begin{align*}
\mathscr{L}(E,F,G)
=&\frac{1}{2}\begin{pmatrix}
\big(\sum_{k,l}\partial_{kl}(E_{kj}^{il}+E_{ki}^{jl})\big)_{i,j=1,...,n+1}\vspace{0.5ex}\\  
\big(\sum_{k,l}\partial_{kl}(F_{kj}^{il}+F_{ki}^{jl}+G_{kj}^{il}-G_{ki}^{jl})\big)_{i,j=1,...,n+1}
\end{pmatrix}.
\end{align*}	

\begin{lemma}\label{lem02}
For any $E\in \mathscr{H}_{1},F\in \mathscr{H}_{2},G\in \mathscr{H}_{3}$, we have $W:=\mathscr{L}(E,F,G)\in \mathscr{M}$ and $\mathrm{div}W=0.$

\end{lemma}

Applying the proofs of Lemmas 3.3 and 3.4 of \cite{DS2009} with a slight modification, we obtain that Lemmas \ref{lem01} and \ref{lem02} hold. We now give the proof of Proposition \ref{pro01}.
\begin{proof}[Proof of Proposition \ref{pro01}]
The proof is divided into two parts as follows.
	
{\bf Part 1.} Consider the case when $\overline{W}\in\mathscr{M}$ satisfies $\overline{W}e_{1}=0$ and $\overline{W}e_{n+1}\neq0$. For simplicity, write $\overline{W}=(\overline{U},\overline{V})^{t}$ for some $\overline{U}\in\mathscr{K}_{1}$ and $\overline{V}\in\mathscr{K}_{2}$. Decompose $\overline{V}$ into two parts as follows:
\begin{align*}
\overline{V}=\overline{V}_{\mathrm{sym}}+\overline{V}_{\mathrm{skew}},
\end{align*}
where
\begin{align*}
\overline{V}_{\mathrm{sym}}\in\left\{\begin{pmatrix}
		0&b\\
		b&0	
	\end{pmatrix}\bigg|\,b\in\mathbb{R}^{n}\right\},\quad \overline{V}_{\mathrm{skew}}\in\left\{\begin{pmatrix}
		Q&0\\
		0&0	
	\end{pmatrix}\bigg|	\,Q\in \mathscr{S}_{-}^{n}\right\}.	
\end{align*}
Since $\overline{W}e_{1}=0$, it is valid to take 
\begin{align*}
\begin{cases}
E_{1i}^{1j}=E_{i1}^{j1}=-E_{i1}^{1j}=-E_{1i}^{j1}=\overline{U}_{ij}\frac{\sin(N y_{1})}{N^{2}},\\
F_{1i}^{1j}=F_{i1}^{j1}=-F_{i1}^{1j}=-F_{1i}^{j1}=(\overline{V}_{\mathrm{sym}})_{ij}\frac{\sin(N y_{1})}{N^{2}},\\
G_{1i}^{1j}=G_{i1}^{j1}=-G_{i1}^{1j}=-G_{1i}^{j1}=(\overline{V}_{\mathrm{skew}})_{ij}\frac{\sin(N y_{1})}{N^{2}},
\end{cases}		
\end{align*}	
and let the rest of the entries be zero. Then we have $E\in \mathscr{H}_{1},F\in \mathscr{H}_{2},G\in \mathscr{H}_{3}$. Choose a smooth cutoff function $\psi$ with the property that $\|\psi\|_{L^{\infty}}\leq1$, $\psi=1$ on $B_{1/2}(0)$, and $\mathrm{supp}\,\psi\subset B_{1}(0)$. Denote $W=\mathscr{L}(\psi E,\psi F,\psi G).$ Then from Lemma \ref{lem02}, we obtain that $W\in\mathscr{M}$ and $\mathrm{div}W=0.$ By a direct computation, we have
\begin{align*}
W(y)=\overline{W}\sin(N y_{1}),\quad\mathrm{in}\; B_{1/2}(0),	
\end{align*}	
which leads to that for any $m>0$,
\begin{align}\label{AQ03}
\dashint_{B_{1}(0)}|W(y)e_{n+1}|^{m}dy\geq\frac{|\overline{W}e_{n+1}|^{m}}{|B_{1}(0)|}\int_{B_{1/2}(0)}|\sin(Ny_{1})|^{m}dy\geq2\alpha|\overline{W}e_{n+1}|^{m},
\end{align}	
for some constant $\alpha=\alpha(n,m)>0$. For any given $\delta>0,$ if $N$ is chosen sufficiently large, we deduce
\begin{align*}
\|W-\psi\mathscr{L}(E,F,G)\|_{L^{\infty}}\leq C\|\psi\|_{C^{2}}\|(E,F,G)\|_{C^{1}}\leq\frac{C\|\psi\|_{C^{2}}}{N}<\delta,	
\end{align*}	
which, together with the fact that $\mathrm{supp}(\psi\mathscr{L}(E,F,G))\subset B_{1}(0)$ and $\psi\mathscr{L}(E,F,G)\in\sigma$, shows that $W$ takes values in the $\delta$-neighborhood of $\sigma$. 

{\bf Part 2.} The general case is treated by reduction to the preceding situation. Observe that there exists $\xi\in\mathbb{R}^{n+1}\setminus\{0\}$ such that $\overline{W}\xi=0$. This, in combination with the fact that $\overline{W}e_{n+1}\neq0$, reads that the vectors $\xi$ and $e_{n+1}$ span a two-dimensional subspace. Pick a basis $\{\eta_{i}\}_{i=1}^{n+1}$ of $\mathbb{R}^{n+1}$ satisfying that $\eta_{1}=\xi$ and $\eta_{n+1}=e_{n+1}$. Then one can find $A\in\mathscr{G}$ such that $Ae_{i}=\eta_{i}$, $i=1,...,n+1$. Let $\widehat{W}=A^{t}\overline{W}A$. Then we obtain 
\begin{align*}
e_{n+1}^{t}\widehat{W}e_{n+1}=(Ae_{n+1})^{t}\overline{W}Ae_{n+1}=e_{n+1}^{t}\overline{W}e_{n+1}=0,	
\end{align*}	 
which gives that $\widehat{W}\in \mathscr{M}.$ Moreover, we have $\widehat{W}e_{1}=0$ and $\widehat{W}e_{n+1}\neq0$. Define the map as follows:
\begin{align*}
T:\,&\mathscr{M}\longrightarrow\mathscr{M},\\
&W\longmapsto A^{-t}WA^{-1}.
\end{align*}
Obviously, $T$ is a bijective linear map. Then by the same arguments as in {\bf Part 1}, we construct a smooth map $\widetilde{W}:\mathbb{R}^{n+1}\rightarrow\mathscr{M}$ compactly supported in $B_{1}(0)$ such that its image is contained in the $\|T\|^{-1}\delta$-neighborhood of the line segment $\tau$ connecting $-\widehat{W}$ and $\widehat{W}$, and $\widetilde{W}(y)=\widehat{W}\sin(Ny_{1})$ in $B_{1/2}(0)$. Let
\begin{align*}
W^{\ast}(y)=A^{-t}\widetilde{W}(A^{t}y)A^{-1}.	
\end{align*}	 
Then we have $\mathrm{supp}\,W^{\ast}\subset A^{-t}B_{1}(0)$. From Lemma \ref{lem01}, we deduce that $\mathrm{div}_{y}W^{\ast}=0$. Since the above isomorphism $T$ transforms the line segment $\tau$ onto $\sigma$, we obtain that the image of $W^{\ast}$ lies in the $\delta$-neighborhood of $\sigma$. Next we make use of standard covering and rescaling arguments to complete the proof. Specifically, there exist finitely many points $y_{k}\in B_{1}(0)$ and radii $r_{k}>0$, $k=1,...,N_{\ast}$ such that the regions $A^{-t}B_{r_{k}}(y_{k})$, after rescaling and translation, are mutually disjoint and entirely lie within $B_{1}(0)$, and
\begin{align*}
\sum^{N_{\ast}}_{k=1}|A^{-t}B_{r_{k}}(y_{k})|\geq\frac{|B_{1}(0)|}{2}.	
\end{align*}	
Define
\begin{align*}
W=\sum^{N_{\ast}}_{k=1}W^{\ast}_{k},\quad\mathrm{with}\;W^{\ast}_{k}(y)=W^{\ast}((y-A^{-t}y_{k})/r_{k}).	
\end{align*}	
Therefore, $W$ is a smooth, divergence-free matrix field with support contained in $B_{1}(0)$ and its values remain within the $\delta$-neighborhood of $\sigma$. Furthermore, for any $m>0$ and $k=1,...,N_{\ast}$, we have
\begin{align*}
&\int_{A^{-t}B_{r_{k}}(y_{k})}|W^{\ast}_{k}(y)e_{n+1}|^{m}dy=\int_{A^{-t}B_{r_{k}}(y_{k})}|A^{-t}\widetilde{W}((A^{t}y-y_{k})/r_{k})e_{n+1}|^{m}dy\notag\\
&=\int_{B_{1}(0)}\frac{r_{k}^{n+1}|A^{-t}\widetilde{W}(y)e_{n+1}|^{m}}{|\det A^{t}|}dy\notag\\
&\geq\frac{|A^{-t}B_{r_{k}}(y_{k})||A^{-t}\widehat{W}e_{n+1}|^{m}}{|B_{1}(0)|}\int_{B_{1/2}(0)}|\sin(Ny_{1})|^{m}dy\notag\\
&\geq2\alpha|A^{-t}B_{r_{k}}(y_{k})||\overline{W}e_{n+1}|^{m},
\end{align*}
where the constant $\alpha$ is given by \eqref{AQ03}. Then we obtain that for $m>0$,
\begin{align*}
\dashint_{B_{1}(0)}|W(y)e_{n+1}|^{m}dy\geq&\sum^{N_{\ast}}_{k=1}\int_{A^{-t}B_{r_{k}}(y_{k})}\frac{|W^{\ast}_{k}(y)e_{n+1}|^{m}}{|B_{1}(0)|}dy\notag\\
\geq&\frac{2\alpha\sum^{N_{\ast}}_{k=1}|A^{-t}B_{r_{k}}(y_{k})|}{|B_{1}(0)|}|\overline{W}e_{n+1}|^{m}\geq\alpha|\overline{W}e_{n+1}|^{m}.	
\end{align*}	
The proof is complete.

\end{proof}

\section{Convex geometry and energy growth}\label{SEC03}

Introduce the following sets
\begin{align*}
K=\{&(u,b,M,Q)\in\mathbb{S}^{n-1}\times\mathbb{S}^{n-1}\times\mathscr{S}_{0}^{n}\times\mathscr{S}_{-}^{n}|\,M=u\otimes u-b\otimes b,\notag\\
&Q=b\otimes u-u\otimes b\},	\quad\mathscr{U}=\mathrm{int}(K^{c_{0}}\times[-1,1]).
\end{align*}	
Here, $K^{c_{0}}$ is the convex hull of $K$ and the notation $\mathrm{int}$ represents the topological interior of the set in $\mathbb{R}^{n}\times\mathbb{R}^{n}\times\mathscr{S}_{0}^{n}\times\mathscr{S}_{-}^{n}$. Clearly, any $(u,b,M,Q,q)$ that solves the relaxation equation of \eqref{TR01} and lies in the convex extremal points of $\overline{\mathscr{U}}$ actually provides a nontrivial solution to the ideal MHD equations. We now prove that $0$ belongs to the interior of $\mathscr{U}$ and thus there exist plenty of localized plane waves contained therein, which are subsequently incorporated into the relaxed solutions to progressively approximate the weak solution of ideal MHD system. 
\begin{lemma}\label{lem03}
For the above-defined sets $K$ and $\mathscr{U}$, we have $0\in\mathrm{int}\,K^{c_{0}}$ and thus $0\in\mathscr{U}$.	
	
\end{lemma}	 

\begin{proof}
Denote by $\mu$ and $\mu\times\mu$ the normalized Haar measure on $\mathbb{S}^{n-1}$ and the corresponding product
measure on $\mathbb{S}^{n-1}\times \mathbb{S}^{n-1}$, respectively. Introduce the following linear map 
\begin{align*}
&T:C(\mathbb{S}^{n-1}\times \mathbb{S}^{n-1})\longrightarrow\mathbb{R}^{n}\times\mathbb{R}^{n}\times\mathscr{S}_{0}^{n}\times\mathscr{S}_{-}^{n},\notag\\
&\phi\longmapsto\int_{\mathbb{S}^{n-1}\times\mathbb{S}^{n-1}}
(u,b,u\otimes u-b\otimes b,b\otimes u-u\otimes b)\,
\phi(u,b)\,d(\mu\times\mu).
\end{align*}	
The first step is to prove that the mapping $T$ is open. In light of the Open Mapping Theorem, it suffices to show that $T$ is surjective. For that purpose,  we only need to demonstrate that the image of $T$ has $n(n+2)$ linearly independent elements, since the dimension of $\mathbb{R}^{n}\times\mathbb{R}^{n}\times\mathscr{S}_{0}^{n}\times\mathscr{S}_{-}^{n}$ is $n(n+2)$. Introduce three positive constants as follows:
\begin{align*}
\gamma_{1}=\int_{\mathbb{S}^{n-1}} u_{1}^{2}d\mu,\quad
\gamma_{2}=\int_{\mathbb{S}^{n-1}} u_{1}^{2} u_{2}^{2}d\mu,\quad\gamma_{3}=\int_{\mathbb{S}^{n-1}}\big(u_{1}^{2}-1/n\big)^{2}d\mu.
\end{align*}
First, we deduce that for $i=1,2,...,n,$
\begin{align*}
T(\phi)=\gamma_{1}
\begin{cases}
( e_{i},0,0,0),&\text{if $\phi(u,b)=u_{i}$,}\\
(0,e_{i},0,0),&\text{if $\phi(u,b)=b_{i}$}.
\end{cases}
\end{align*}
Second, by choosing $\phi(u,b)=u_{i}u_{j}$ or $\phi(u,b)=-b_{i}b_{j}$ with $1\leq i<j\leq n$, we obtain
\begin{align*}
T(\phi)=\gamma_{2}(0,0,e_{i}\otimes e_{j}+e_{j}\otimes e_{i},0).
\end{align*}
Third, taking $\phi(u,b)=u_{i}^{2}-1/n$ or $\phi(u,b)=-b_{i}^{2}+1/n$ with $i=1,2,...n,$ we have
\begin{align*}
T(\phi)=\gamma_3\Big(0,0,e_i\otimes e_i-\frac{1}{n-1}\sum_{j\neq i}e_{j}\otimes e_{j},0\Big).
\end{align*}
Finally, picking $\phi(u,b)=u_{i} b_{j}$ with $1\leq i<j\leq n$, we derive
\begin{align*}
T(\phi)= \gamma_{1}^{2}(0,0,0,e_{j}\otimes e_{i}-e_{i}\otimes e_{j}).
\end{align*}
These above facts imply that $T$ is surjective.

Next we proceed to prove $0\in\mathrm{int}\,K^{c_{0}}$ by using the open mapping property of $T$. Define two sets as follows: 
\begin{align*}
G=\big\{\psi|\,\|\psi\|_{C(\mathbb{S}^{n-1}\times\mathbb{S}^{n-1})}<1/2\big\},\quad\widetilde{G}=\{\beta_{\psi}+\psi|\,\psi\in G\},
\end{align*}
where $\beta_{\psi}=1-\int_{\mathbb{S}^{n-1}\times\mathbb{S}^{n-1}}\psi d(\mu\times\mu).$ Note that if $\phi\geq0$ and $\int_{\mathbb{S}^{n-1}\times\mathbb{S}^{n-1}}\phi d(\mu\times\mu)=1$, we have $T(\phi)\in K^{c_{0}}$. It then follows that $T(\widetilde{G})\subset K^{c_{0}}$. This, together with the fact that $T(a_{0})=0$ for any fixed constant $a_{0}\in\mathbb{R}$, we deduce that $0=T(0)\in T(G)\subset K^{c_{0}}.$ Since $T$ is an open mapping, we then have $0\in\mathrm{int}\,K^{c_{0}}$ and $0\in\mathscr{U}$.

\end{proof}

\begin{definition}
Fix a constant $\lambda\in(0,1)$. For any given $z\in\mathscr{U}$, if there exist $\bar{z}\in\Lambda\setminus\{0\}$ and $r>0$ such that 
\begin{align*}
\sigma\subset \mathscr{U},\quad\mathrm{dist}(\sigma,\partial\mathscr{U})\geq\lambda\mathrm{dist}(z,\partial\mathscr{U}),\quad\text{with }\sigma:=\{z+s\bar{z}\,|\,|s|\leq r\},
\end{align*}		
then we say that $\mathscr{U}$ possesses $\lambda$-relaxed local convexity at $z$ along the wave-cone direction $\bar{z}$.	
\end{definition}	

Next we show that the set $\mathscr{U}$ enjoys $1/2$-relaxed local convexity in the wave-cone directions.
\begin{lemma}\label{lem05}
Let $n\geq4$. For any $(u,b,M,Q,q)\in\mathscr{U}$, there exists $(\bar{u},\bar{b},\overline{M},\overline{Q},0)\in\Lambda$ such that $\mathscr{U}$ has $1/2$-relaxed local convexity at $(u,b,M,Q,q)$ in the wave-cone direction $(\bar{u},\bar{b},\overline{M},\overline{Q},0)$. Moreover, for any $\varepsilon\in(0,1)$ and $m\in(1,\infty)$, we have
\begin{align}\label{AE03}
\frac{\varepsilon^{m-1}}{C_{0}}(2-|u|^{m}-|b|^{m})\leq	|(\bar{u},\bar{b})|^{m}+\frac{2\varepsilon^{m}}{C_{0}},
\end{align}		
for some constant $C_{0}=C_{0}(n,m)>0$.
\end{lemma}
\begin{remark}
Lemma \ref{lem03} ensures that the set $\mathscr{U}$ has a nonempty interior, thereby guaranteeing the validity of the conclusions of Lemma \ref{lem05}.
\end{remark}

Before proving Lemma \ref{lem05}, we first recall the following elementary inequality.
\begin{lemma}\label{lem06}	
For any $0<\varepsilon<1$, $m>1$ and $a,c>0$, there holds
\begin{align}\label{AQ06}
(a+c)^{m}\leq(1+\varepsilon)a^{m}+\frac{C(m)}{\varepsilon^{m-1}}c^{m},
\end{align}
for some constant $C(m)>0$.
\end{lemma}	
\begin{remark}
Based on the binomial theorem and Young’s inequality, the proof of Lemma \ref{lem06} was previously given in Lemma 2.1 of \cite{MZ2024} in the case when $m$ is an integer. In fact, the inequality in \eqref{AQ06} also remains true for non-integer values of $m$. This fact can be obtained from the convexity of the power function. To be specific, by letting $\theta=1-(1+\varepsilon)^{-1/(m-1)}$, it follows from a direct calculation that for any $0<\varepsilon<1$,
\begin{align*}
(a+c)^{m}\leq(1-\theta)^{1-m}a^{m}+\theta^{1-m}c^{m}\leq(1+\varepsilon)a^{m}+2^{m}(m-1)^{m-1}\varepsilon^{1-m}c^{m}.
\end{align*}	

\end{remark}	

We are now ready to give the proof of Lemma \ref{lem05}.
\begin{proof}[Proof of Lemma \ref{lem05}]
For simplicity, denote $z'=(u,b,M,Q)\in \mathrm{int}\,K^{c_{0}}$ and $z=(z',q)\in \mathscr{U}$.\ Making use of Carath\'{e}odory’s theorem, we see that there exist $\{\theta_{k}\}_{k=1}^{N}\subset(0,1)$ and $z'_{k}=(u_{k},b_{k},M_{k},Q_{k})\in K$ such that
\begin{align*}
z'=\sum^{N}_{k=1}\theta_{k}z'_{k},\quad\sum^{N}_{k=1}\theta_{k}=1,
\end{align*}
where $N=n(n+2)$. Without loss of generality, we assume that $\theta_{1}$ is the largest coefficient in the set $\{\theta_{k}\}_{k=1}^{N}$. Fix $l>1$ such that 
\begin{align*}
\theta_{l}|(u_{l},b_{l})-(u_{1},b_{1})|=\max_{2\leq k\leq N}\theta_{k}|(u_{k},b_{k})-(u_{1},b_{1})|.
\end{align*}
Define 
\begin{align*}
\bar{z}'=(\bar{u},\bar{b},\overline{M},\overline{Q})=\frac{1}{2}\theta_{l}(z'_{l}-z'_{1}).	
\end{align*}	
Then we have from a straightforward computation that
\begin{align*}
z'\pm \bar{z}'\in\mathrm{int}\,K^{c_{0}},\quad \{z'+s\bar{z}'|\,|s|\leq2\}\subset K^{c_{0}}.	
\end{align*}	
Denote $\bar{z}=(\bar{z}',0)$. Then we also have $\sigma:=\{z+s\bar{z}|\,|s|\leq1\}\subset \mathscr{U}$ and $\{z+s\bar{z}|\,|s|\leq2\}\subset \overline{\mathscr{U}}$. Since $\mathscr{U}$ is convex, we assume without loss of generality that 
\begin{align*}
\mathrm{dist}(z+\bar{z},\partial\mathscr{U})=\mathrm{dist}(\sigma,\partial\mathscr{U}).
\end{align*}	
Then there exists some $\tilde{z}\in\partial\mathscr{U}$ such that $\mathrm{dist}(z+\bar{z},\tilde{z})=\mathrm{dist}(\sigma,\partial\mathscr{U})$. Define two segments as follows:
\begin{align*}
\tau_{1}=\{s(z+\bar{z})+(1-s)\tilde{z}|\,s\in[0,1]\},\quad\tau_{2}=\{s(z+2\bar{z})+(1-s)\tilde{z}|\,s\in[0,1]\}.	
\end{align*}	
Starting from the point $z$, we draw the ray parallel to the segment $\tau_{1}$.\ Then we extend the segment $\tau_{2}$ and denote by $\hat{z}$ the intersection point of these two rays. Then we obtain
\begin{align*}
\frac{1}{2}=\frac{\mathrm{dist}(z+w,\tilde{z})}{\mathrm{dist}(z,\hat{z})}\leq\frac{\mathrm{dist}(\sigma,\partial\mathscr{U})}{\mathrm{dist}(z,\partial\mathscr{U})}.	
\end{align*}	
Moreover, it follows from Lemma \ref{lem06} that for any $\varepsilon\in(0,1)$,
\begin{align*}
2-|u|^{m}-|b|^{m}\leq&\frac{C(m)}{\varepsilon^{m-1}}((1-|u|)^{m}+(1-|b|)^{m})+\varepsilon(|u|^{m}+|b|^{m})\notag\\
\leq&\frac{C(m)}{\varepsilon^{m-1}}(|u-u_{1}|^{m}+|b-b_{1}|^{m})+2\varepsilon\notag\\
\leq&\frac{NC(m)}{\varepsilon^{m-1}}\max_{1\leq k\leq N}\theta_{k}^{m}|(u_{k}-u_{1},b_{k}-b_{1})|^{m}+2\varepsilon\notag\\
\leq&\frac{2^{m}NC(m)}{\varepsilon^{m-1}}|(\bar{u},\bar{b})|^{m}+2\varepsilon,
\end{align*}
which implies that \eqref{AE03} holds.

It remains to prove that $(\bar{u},\bar{b},\overline{M},\overline{Q},0)\in\Lambda$. For that purpose, it only needs to demonstrate that for any given $u_{i},b_{i}\in\mathbb{S}^{n-1}$, $i=1,2,$ there exists $\xi\in\mathbb{R}^{n+1}\setminus\{0\}$ such that
\begin{align}\label{DE01}
	\begin{pmatrix}
		u_{1}\otimes u_{1}-b_{1}\otimes b_{1}-u_{2}\otimes u_{2}+b_{2}\otimes b_{2} & u_{1}-u_{2} \\
		u_{1}-u_{2} & 0 \\
		b_{1}\otimes u_{1}-u_{1}\otimes b_{1}-b_{2}\otimes u_{2}+u_{2}\otimes b_{2} & b_{1}-b_{2} \\
		b_{1}-b_{2} & 0
	\end{pmatrix}\xi=0.
\end{align}
Denote $\xi=(\zeta,s)$, where $\zeta\in\mathbb{R}^{n}.$ Then \eqref{DE01} becomes
\begin{align}\label{DE02}
\begin{cases}	
(u_{1}\otimes u_{1}-b_{1}\otimes b_{1}-u_{2}\otimes u_{2}+b_{2}\otimes b_{2})\zeta +s(u_{1}-u_{2})=0, \\
(b_{1}\otimes u_{1}-u_{1}\otimes b_{1}-b_{2}\otimes u_{2}+u_{2}\otimes b_{2})\zeta+s(b_{1}-b_{2})=0,\\
(u_{1}-u_{2})\cdot\zeta=(b_{1}-b_{2})\cdot\zeta=0.
\end{cases}
\end{align}
First, we restrict the range of $\zeta\in\mathrm{span}\{u_{1}-u_{2},b_{1}-b_{2}\}^{\perp}$ so as to ensure that the third line of \eqref{DE02} holds. Then by letting $c_{1}=u_{1}\cdot\zeta=u_{2}\cdot\zeta$ and $c_{2}=b_{1}\cdot\zeta=b_{2}\cdot\zeta$, we have from \eqref{DE02} that
\begin{align}\label{DE03}
	(c_{1}+s)(u_{1}-u_{2})-c_{2}(b_{1}-b_{2})=0,\quad (c_{1}+s)(b_{1}-b_{2})-c_{2}(u_{1}-u_{2})=0.
\end{align}
Observe that when $n\geq4$, the linear space $\mathrm{span}\{u_{1}-u_{2},b_{1},b_{2}\}^{\perp}$ must be nontrivial. Therefore, by taking any element $\zeta\in\mathrm{span}\{u_{1}-u_{2},b_{1},b_{2}\}^{\perp}\setminus\{0\}$ and then picking $s=-u_{1}\cdot\zeta$, we obtain
\begin{align*}
c_{1}+s=c_{2}=0.
\end{align*}	
This implies that \eqref{DE03} holds and thus we find nonzero $\xi=(\zeta,s)$ satisfying \eqref{DE01} in the case when $n\ge4$.

\end{proof}

Denote by $X_{0}$ the set of smooth relaxed solutions, whose elements $(u,b,M,Q,q)\in C^{\infty}_{x,t}$ satisfy the following constraints:
\begin{align*}
\begin{cases}
\mathrm{supp}(u,b,M,Q,q)\subset\Omega,\\
\text{$(u,b,M,Q,q)$ solves \eqref{TR01} in $\mathbb{R}_{x}^{n}\times\mathbb{R}_{t}$,}\\
\text{$(u,b,M,Q,q)(x,t)\in\mathscr{U},$ $\forall\,(x,t)\in\mathbb{R}_{x}^{n}\times\mathbb{R}_{t},$}
\end{cases}	
\end{align*}	
where $\Omega$ is a given bounded domain. Remark that the elements of $X_{0}$ possess good properties, such as boundedness, smoothness, and compact support. Consequently, for any element from this set, we can superimpose high-frequency localized plane-wave perturbations to generate a sequence of relaxed solutions in its neighborhood. This sequence both increases the energy and maintains the element as its weak convergence limit. The detailed statement is presented as follows.
\begin{lemma}\label{lem10}
For any fixed $(u_{\ast},b_{\ast},M_{\ast},Q_{\ast},q_{\ast})\in X_{0}$, we can find a sequence $\{(u_{k},b_{k},M_{k},Q_{k},q_{k})\}\subset X_{0}$ such that 
\begin{align}\label{AQ15}
(u_{k},b_{k},M_{k},Q_{k},q_{k})\overset{\ast}{\xrightharpoonup{\quad}}(u_{\ast},b_{\ast},M_{\ast},Q_{\ast},q_{\ast}) \quad\mathrm{in}\;L^{\infty}(\Omega),
\end{align}	
and, for any $m\geq2$,
\begin{align}\label{AQ17}
&\liminf_{k\rightarrow\infty}\big(\|u_{k}\|_{L^{m}(\Omega)}^{m}+\|b_{k}\|_{L^{m}(\Omega)}^{m}\big)\notag\\
&\geq\|u_{\ast}\|_{L^{m}(\Omega)}^{m}+\|b_{\ast}\|_{L^{m}(\Omega)}^{m}+\beta\big(2|\Omega|-\|u_{\ast}\|_{L^{m}(\Omega)}^{m}-\|b_{\ast}\|_{L^{m}(\Omega)}^{m}\big)^{m},
\end{align}
where the constant $\beta=\beta(n,m,|\Omega|)>0$.		
\end{lemma}	

\begin{remark}
This energy growth lemma essentially reflects the core mechanism of the convex integration method developed by Gromov \cite{G1973,G1986}, namely the iterative insertion of localized, highly oscillatory perturbations to construct a sequence of smooth relaxed solutions approximating a weak solution of the considered equations.\ The early prototype of convex integration can be traced back to Nash’s iterative technique \cite{N1954} in the study of the isometric embedding problem.	

\end{remark}	

\begin{proof}
{\bf Part 1.} First, we have from Lemma \ref{lem05} that for each point $(x,t)\in\Omega$, one can find a direction $(\bar{u},\bar{b},\overline{M},\overline{Q},0)\in\Lambda$ such that for any $\varepsilon\in(0,1)$, 
\begin{align}\label{AE08}
	\frac{\varepsilon^{m-1}}{C_{0}}(2-|u_{\ast}|^{m}-|b_{\ast}|^{m})\leq	|(\bar{u},\bar{b})|^{m}+\frac{2\varepsilon^{m}}{C_{0}},
\end{align}		
and 
\begin{align}\label{AQE09}
	\sigma\subset \mathscr{U},\quad\mathrm{dist}(\sigma,\partial\mathscr{U})\geq\frac{1}{2}\mathrm{dist}((u_{\ast},b_{\ast},M_{\ast},Q_{\ast},q_{\ast}),\partial\mathscr{U}),
\end{align}		
where $C_{0}=C_{0}(n,m)>0$, and
\begin{align*}
\sigma=\big\{(u_{\ast},b_{\ast},M_{\ast},Q_{\ast},q_{\ast})+s(\bar{u},\bar{b},\overline{M},\overline{Q},0)|\,|s|\leq1\big\}.
\end{align*}	
In light of the uniform continuity of $(u_{\ast},b_{\ast},M_{\ast},Q_{\ast},q_{\ast})$, we can take a small constant 
\begin{align*}
0<\delta\leq
\begin{cases}
\frac{\mathrm{dist}((u_{\ast},b_{\ast},M_{\ast},Q_{\ast},q_{\ast}),\partial\mathscr{U})}{4\|(u_{\ast},b_{\ast},M_{\ast},Q_{\ast},q_{\ast})\|_{C^{1}}},&\text{if $(u_{\ast},b_{\ast},M_{\ast},Q_{\ast},q_{\ast})\neq0$},\\
1,&\text{if $(u_{\ast},b_{\ast},M_{\ast},Q_{\ast},q_{\ast})=0$}
\end{cases} 
\end{align*} 
such that for any two points $(x,t),(x_{\ast},t_{\ast})\in\Omega$ with $(x,t)\in B_{\delta}(x_{\ast},t_{\ast})$, 
\begin{align*}
	\tilde{\sigma}\subset \mathscr{U},\quad\mathrm{dist}(\tilde{\sigma},\partial\mathscr{U})\geq\frac{1}{4}\mathrm{dist}((u_{\ast},b_{\ast},M_{\ast},Q_{\ast},q_{\ast}),\partial\mathscr{U}),
\end{align*}		
where
\begin{align*}
\tilde{\sigma}:=\big\{(u_{\ast},b_{\ast},M_{\ast},Q_{\ast},q_{\ast})(x,t)+s(\bar{u},\bar{b},\overline{M},\overline{Q},0)(x_{\ast},t_{\ast})|\,|s|\leq1\big\}\subset\mathscr{U}.	
\end{align*}	
Applying Proposition \ref{pro01} with $(\bar{u},\bar{b},\overline{M},\overline{Q},0)(x_{\ast},t_{\ast})\in\Lambda$, we find a smooth solution $(u,b,M,Q,q)$ of \eqref{TR01} satisfying the required properties as follows:
\begin{align*}
\begin{cases}
\mathrm{supp}(u,b,M,Q,q)\subset B_{1}(0),\;\,\sup\limits_{y\in B_{1}(0)}\mathrm{dist}((u,b,M,Q,q)(y),\hat{\sigma})<\delta,\\
\dashint_{B_{1}(0)}|(u,b)|^{m}dy\geq\alpha|(\bar{u},\bar{b})|^{m},	
\end{cases}
\end{align*}
where the segment $\hat{\sigma}$ is defined by 
\begin{align*}
\hat{\sigma}=\{(2s-1)(\bar{u},\bar{b},\overline{M},\overline{Q},0)(x_{\ast},t_{\ast})|\,s\in[0,1]\}.	
\end{align*}	
For any $0<r<\delta$, after performing the following scaling transformation
\begin{align*}
(u_{r},b_{r},M_{r},Q_{r},q_{r})(x,t)=(u,b,M,Q,q)((x-x_{\ast})/r,(t-t_{\ast})/r),	
\end{align*}	
we see that the rescaled solution $(u_{r},b_{r},M_{r},Q_{r},q_{r})$ satisfies 
\begin{align}\label{AE18}
	\begin{cases}
		\mathrm{supp}(u_{r},b_{r},M_{r},Q_{r},q_{r})\subset B_{r}(x_{\ast},t_{\ast}),\\
		\sup\limits_{(x,t)\in B_{r}(x_{\ast},t_{\ast})}\mathrm{dist}((u_{r},b_{r},M_{r},Q_{r},q_{r})(x,t),\hat{\sigma})<\delta,\\
		\dashint_{B_{r}(x_{\ast},t_{\ast})}|(u_{r},b_{r})|^{m}dxdt\geq\alpha|(\bar{u},\bar{b})|^{m},	
	\end{cases}
\end{align}
which, together with \eqref{AQE09}, shows that
\begin{align}\label{AE19}
\{(u_{\ast},b_{\ast},M_{\ast},Q_{\ast},q_{\ast})+(u_{r},b_{r},M_{r},Q_{r},q_{r})|\,0<r<\delta\}\subset X_{0}.	
\end{align}

{\bf Part 2.} By finite Vitali covering lemma, we obtain that for any small $0<\varsigma,\kappa<1$, one can extract $\mathcal{N}=\mathcal{N}(\varsigma,\kappa)$ mutually disjoint balls $\{B_{r_{l}}(x_{l},t_{l})\}_{l=1}^{\mathcal{N}}\subset\Omega$ with the property that
$0<r_{l}<\kappa,\,\,l=1,...,\mathcal{N},\, |\Omega\setminus\scaleto{\cup}{7pt}_{l=1}^{\mathcal{N}}B_{r_{l}}(x_{l},t_{l})|<\varsigma.$
Owing to the uniform continuity and boundedness of $(u_{\ast},b_{\ast})$, we deduce that for any fixed $\varsigma\in(0,1)$, there exists some $\kappa_{0}\in(0,1)$ such that if $r_{l}\in(0,\kappa_{0})$, for any $m\geq2$ and $(x,t)\in B_{r_{l}}(x_{l},t_{l})$,
\begin{align*}
&|(|u_{\ast}|^{m}+|b_{\ast}|^{m})(x,t)-(|u_{\ast}|^{m}+|b_{\ast}|^{m})(x_{l},t_{l})|\notag\\
&\leq m\max\{|u_{\ast}(x,t)|^{m-1},|u_{\ast}(x_{l},t_{l})|^{m-1}\}|u_{\ast}(x,t)-u_{\ast}(x_{l},t_{l})|\notag\\
&\quad+m\max\{|b_{\ast}(x,t)|^{m-1},|b_{\ast}(x_{l},t_{l})|^{m-1}\}|b_{\ast}(x,t)-b_{\ast}(x_{l},t_{l})|\leq2m\varsigma.
\end{align*}	
Then we have
\begin{align}\label{AE16}
\int_{\Omega}(2-|u_{\ast}|^{m}-|b_{\ast}|^{m})\leq&\sum^{\mathcal{N}}_{l=1}(2-|u_{\ast}(x_{l},t_{l})|^{m}-|b_{\ast}(x_{l},t_{l})|^{m})|B_{r_{l}}(x_{l},t_{l})|\notag\\
&+2m\varsigma\sum^{\mathcal{N}}_{l=1}|B_{r_{l}}(x_{l},t_{l})|+2\big|\Omega\setminus\scaleto{\cup}{7pt}_{l=1}^{\mathcal{N}}B_{r_{l}}(x_{l},t_{l})\big|\notag\\\leq&2\sum^{\mathcal{N}}_{l=1}\big(2-|u_{\ast}(x_{l},t_{l})|^{m}-|b_{\ast}(x_{l},t_{l})|^{m}\big)|B_{r_{l}}(x_{l},t_{l})|,
\end{align}	
provided that the constant $\varsigma$ is chosen to satisfy
\begin{align*}
0<\varsigma\leq	\frac{|\Omega|\min\limits_{(x,t)\in\overline{\Omega}}(2-|u_{\ast}|^{m}-|b_{\ast}|^{m})}{4(1+m|\Omega|)}.
\end{align*}	 
For any $k\in\mathbb{N}$ with $1/k<\min\{\delta,\kappa_{0}\}$, we can select $\mathcal{N}_{k}=\mathcal{N}(k,\varsigma)$ mutually disjoint balls $B_{r_{k,l}}(x_{k,l},t_{k,l})\subset\Omega$ with radii less than $1/k$ such that \eqref{AE16} is satisfied. On every such ball $B_{r_{k,l}}(x_{k,l},t_{k,l})$, the  construction in {\bf Part 1} can be applied to obtain $(u_{k,l},b_{k,l},M_{k,l},Q_{k,l},q_{k,l})$ with the same properties as in \eqref{AE18}--\eqref{AE19}. Let
\begin{align*}
(u_{k},b_{k},M_{k},Q_{k},q_{k})=(u_{\ast},b_{\ast},M_{\ast},Q_{\ast},q_{\ast})+\sum_{l=1}^{\mathcal{N}_{k}}(u_{k,l},b_{k,l},M_{k,l},Q_{k,l},q_{k,l}).
\end{align*}	
By picking $\varepsilon=\frac{1}{8}\dashint_{\Omega}\big(2-|u_{\ast}|^{m}-|b_{\ast}|^{m}\big)dxdt$ in \eqref{AE08}, we have from \eqref{AE18}--\eqref{AE16} that $(u_{k},b_{k},M_{k},Q_{k},q_{k})\in X_{0}$, and
\begin{align}\label{AQE19}
&\int_{\Omega}|(u_{k},b_{k})-(u_{\ast},b_{\ast})|^{m}dxdt=\sum^{\mathcal{N}_{k}}_{l=1}\int_{B_{r_{k,l}}(x_{k,l},t_{k,l})}|(u_{k,l},b_{k,l})|^{m}dxdt\notag\\
&\geq\alpha\sum^{\mathcal{N}_{k}}_{l=1}|(\bar{u},\bar{b})(x_{k,l},t_{k,l})|^{m}|B_{r_{k,l}}(x_{k,l},t_{k,l})|\notag\\
&\geq\frac{\alpha\varepsilon^{m-1}}{C_{0}}\sum^{\mathcal{N}_{k}}_{l=1}\big(2-(|u_{\ast}|^{m}+|b_{\ast}|^{m})(x_{k,l},t_{k,l})\big)|B_{r_{k,l}}(x_{k,l},t_{k,l})|-\frac{2\alpha\varepsilon^{m}|\Omega|}{C_{0}}\notag\\
&\geq\frac{\alpha\varepsilon^{m-1}}{2C_{0}}\int_{\Omega}\big(2-|u_{\ast}|^{m}-|b_{\ast}|^{m}\big)dxdt-\frac{2\alpha\varepsilon^{m}|\Omega|}{C_{0}}\notag\\	
&\geq\frac{\alpha}{4^{2m-1}C_{0}|\Omega|^{m-1}}\bigg(\int_{\Omega}\big(2-|u_{\ast}|^{m}-|b_{\ast}|^{m}\big)dxdt\bigg)^{m}.
\end{align}	
From the proof of Proposition \ref{pro01}, we have 
\begin{align*}
\int_{\mathbb{R}^{n}_{x}\times\mathbb{R}_{t}}(u_{k,l},b_{k,l},M_{k,l},Q_{k,l},q_{k,l})dxdt=0,\quad l=1,2,...,\mathcal{N}_{k},	
\end{align*}	
which, in combination with the boundedness of $(u_{k,l},b_{k,l},M_{k,l},Q_{k,l},q_{k,l})$, leads to that \eqref{AQ15} holds. Combining \eqref{AQ15} and \eqref{AQE19}, we deduce that for any $m\geq2$,
\begin{align*}
&\liminf_{k\rightarrow\infty}\big(\|u_{k}\|_{L^{m}(\Omega)}^{m}+\|b_{k}\|_{L^{m}(\Omega)}^{m}\big)\notag\\
&\geq\|u_{\ast}\|_{L^{m}(\Omega)}^{m}+\|b_{\ast}\|_{L^{m}(\Omega)}^{m}+2^{2-m}\liminf_{k\rightarrow\infty}\big(\|u_{k}-u_{\ast}\|_{L^{m}(\Omega)}^{m}+\|b_{k}-b_{\ast}\|_{L^{m}(\Omega)}^{m}\big)\notag\\
&\quad+m\liminf_{k\rightarrow\infty}\int_{\Omega}\big(|u_{\ast}|^{m-2}u_{\ast}\cdot(u_{k}-u_{\ast})+|b_{\ast}|^{m-2}b_{\ast}\cdot(b_{k}-b_{\ast})\big)dxdt\notag\\
&\geq\|u_{\ast}\|_{L^{m}(\Omega)}^{m}+\|b_{\ast}\|_{L^{m}(\Omega)}^{m}+2^{3-2m}\liminf_{k\rightarrow\infty}\|(u_{k}-u_{\ast},b_{k}-b_{\ast})\|_{L^{m}(\Omega)}^{m}\notag\\
&\geq\|u_{\ast}\|_{L^{m}(\Omega)}^{m}+\|b_{\ast}\|_{L^{m}(\Omega)}^{m}+\frac{\alpha}{2^{6m-5}C_{0}|\Omega|^{m-1}}\big(2|\Omega|-\|u_{\ast}\|_{L^{m}(\Omega)}^{m}-\|b_{\ast}\|_{L^{m}(\Omega)}^{m}\big)^{m},
\end{align*}	
which reads that \eqref{AQ17} holds.

\end{proof}

\section{The proofs of Theorems \ref{Main01} and \ref{Main02}}\label{SEC04}

Endowing $X_{0}$ with the topology of $L^{\infty}$ weak* convergence of $(u,b,M,Q,q)$, we set $X$ to be the closure of $X_{0}$ with respect to this topology. Combining Lemma \ref{lem03}, we see that $X$ is a nonempty bounded closed subset of $L^{\infty}(\Omega)$. Then the set $X$ is compact in the weak* topology, and thus metrizable. 
\begin{lemma}\label{lem08}
If an element $(u,b,M,Q,q)\in X$ satisfies $|u(x,t)|=1$ or $|b(x,t)|=1$ for almost all $(x,t)\in\Omega$, then $(u,b)$ and $p:=q$ forms a weak solution of  the ideal MHD equations in $\mathbb{R}_{x}^{n}\times\mathbb{R}_{t}$ with $(u,b,p)$ vanishing outside $\Omega$.	
	
\end{lemma}	

\begin{proof}
Due to the fact that $\overline{\mathscr{U}}$ is compact and convex, each $(u,b,M,Q,q)\in X$ provides a solution to \eqref{TR01} with support in $\overline{\Omega}$ and	its image is contained in $\overline{\mathscr{U}}$. Especially, there has $(u,b,M,Q)(x,t)\in K^{c_{0}}$ for almost every $(x,t)\in\mathbb{R}^{n}_{x}\times\mathbb{R}_{t}$. For simplicity, write $z'=(u,b,M,Q)$.\ Then using Carath\'{e}odory's theorem, we have $z'=\sum^{n(n+2)}_{k=1}\lambda_{k}z_{k}'$ for some $z_{k}'=(u_{k},b_{k},M_{k},Q_{k})\in K$, $\lambda_{k}\in[0,1]$ with $\sum^{n(n+2)}_{k=1}\lambda_{k}=1$. In particular, $(u,b)=\sum^{n(n+2)}_{k=1}\lambda_{k}(u_{k},b_{k})$. Since $|u|=1$ or $|b|=1$ almost everywhere in $\Omega$, there exists some $k_{0}\in[1,n(n+2)]$ such that $\lambda_{k_{0}}=1$ and $\lambda_{k}=0$ for $k\neq k_{0}$. Then we obtain $z'=z_{k_{0}}'\in K$, which implies that Lemma \ref{lem08} holds.
	
\end{proof}

From Lemma \ref{lem08}, we reduce the problem of constructing nontrivial solutions of the ideal MHD system to finding the elements in $X$ satisfying that $|u(x,t)|=1$ or $|b(x,t)|=1$ for almost every $(x,t)\in\Omega$.\ In what follows, we carry out this objective by making use of Baire category method and convex integration scheme, respectively.

\subsection{Baire category method}
We endow $X$ with a metric $d^{\ast}_{\infty}$ that induces the weak* topology of $L^{\infty}$, thereby making $(X,d_{\infty}^{\ast})$ be a complete metric space. For any $m\geq2,$ introduce the identity map $I_{m}$ as follows:
\begin{align*}
I_{m}:(X,d^{\ast}_{\infty})&\longrightarrow L^{m}(\mathbb{R}_{x}^{n}\times\mathbb{R}_{t}),\notag\\
(u,b,M,Q,q)&\longmapsto (u,b,M,Q,q).
\end{align*}
\begin{lemma}\label{lem09}
For every $m\geq2$, the operator $I_{m}$ belongs to the class of Baire-1 functions and thus its continuity points constitute a residual set in $(X,d^{\ast}_{\infty})$.

\end{lemma}

\begin{proof}
Define the approximation maps as follows: for any $r>0$ and $m\geq2$,	
\begin{align*}
I_{m,r}:(X,d^{\ast}_{\infty})&\longrightarrow L^{m}(\mathbb{R}_{x}^{n}\times\mathbb{R}_{t}),\notag\\
(u,b,M,Q,q)&\longmapsto (u,b,M,Q,q)\ast\rho_{r},	
\end{align*}	
where $\rho_{r}=r^{-(n+1)}\rho(rx,rt)$ denotes the standard mollifying kernel on $\mathbb{R}^{n+1}$. Then we obtain that for any $(u,b,M,Q,q)\in X,$
\begin{align*}
I_{m,r}(u,b,M,Q,q)\longrightarrow I_{m}(u,b,M,Q,q)\quad\mathrm{in}\;L^{m},\quad \text{as $r\rightarrow0$.}	
\end{align*}		
For any given $r>0$, if $\{(u_{k},b_{k},M_{k},Q_{k},q_{k})\}\subset X$ satisfies 
\begin{align*}
	(u_{k},b_{k},M_{k},Q_{k},q_{k})\overset{\ast}{\xrightharpoonup{\quad}}(u,b,M,Q,q) \quad\mathrm{in}\;L^{\infty},
\end{align*}	
for some $(u,b,M,Q,q)\in X$, then we have
\begin{align*}
I_{m,r}(u_{k},b_{k},M_{k},Q_{k},q_{k})\longrightarrow I_{m,r}(u,b,M,Q,q),\quad \text{for a.e. $(x,t)\in\mathbb{R}_{x}^{n}\times\mathbb{R}_{t}$},	
\end{align*}		
as $k\rightarrow\infty$. This, in combination with Egorov's theorem and Fatou's lemma, gives that
\begin{align}\label{AQ09}
	I_{m,r}(u_{k},b_{k},M_{k},Q_{k},q_{k})\longrightarrow I_{m,r}(u,b,M,Q,q)\quad \text{in $L^{m}$,\quad as $k\rightarrow\infty$}.	
\end{align}
These above facts, together with Lemma 7.3 in \cite{O1980}, show that Lemma \ref{lem09} holds.	
	
\end{proof}	
A consequence of Lemma \ref{lem09}, H\"{o}lder's inequality, the boundedness and compact support of functions in $(X,d^{\ast}_{\infty})$ gives the following result.
\begin{corollary}\label{Cor09}
For any $m>2$, the set of continuity points of $I_{m}$ coincides with that of $I_{2}$.
\end{corollary}

We are now ready to give the proof of Theorem \ref{Main01}.
\begin{proof}[Proof of Theorem \ref{Main01}]
Combining Lemmas \ref{lem08} and \ref{lem09} with Corollary \ref{Cor09}, it suffices to prove that for any fixed $m\geq2$, if $(u,b,M,Q,q)$ is a continuity point of $I_{m}$, then we have $|u(x,t)|=|b(x,t)|=1$ for almost all $(x,t)\in\Omega$. Since $I_{m}$ is continuous at the point $(u,b,M,Q,q)$, we can choose a weak*-convergent sequence $\{(u_{k},b_{k},M_{k},Q_{k},q_{k})\}\subset X_{0}$ approaching $(u,b,M,Q,q)$. With the aid of Lemma \ref{lem10} and a standard diagonalization, we can extract another sequence $(\tilde{u}_{k},\tilde{b}_{k},\widetilde{M}_{k},\widetilde{Q}_{k},\tilde{q}_{k})$ converging weakly* to the same limit and satisfying
\begin{align*}
&\|u\|_{L^{m}(\Omega)}^{m}+\|b\|_{L^{m}(\Omega)}^{m}=\liminf_{k\rightarrow\infty}\big(\|\tilde{u}_{k}\|_{L^{m}(\Omega)}^{m}+\|\tilde{b}_{k}\|_{L^{m}(\Omega)}^{m}\big)\notag\\
&\geq\liminf_{k\rightarrow\infty}\Big(\|u_{k}\|_{L^{m}(\Omega)}^{m}+\|b_{k}\|_{L^{m}(\Omega)}^{m}+\frac{\beta}{2}\big(2|\Omega|-\|u_{k}\|_{L^{m}(\Omega)}^{m}-\|b_{k}\|_{L^{m}(\Omega)}^{m}\big)^{m}\Big)\notag\\	
&\geq\|u\|_{L^{m}(\Omega)}^{m}+\|b\|_{L^{m}(\Omega)}^{m}+\frac{\beta}{2}\big(2|\Omega|-\|u\|_{L^{m}(\Omega)}^{m}-\|b\|_{L^{m}(\Omega)}^{m}\big)^{m}.
\end{align*}	
This, together with the fact of $|u|\leq1$ and $|b|\leq1$ almost everywhere in $\Omega$, leads to that $|u|=|b|=1$ a.e. in $\Omega$ and therefore Theorem \ref{Main01} is proved.

\end{proof}

\subsection{Convex integration scheme}

In this subsection, we present a more straightforward constructive proof of Theorem \ref{Main01} by employing the convex integration scheme developed in \cite{DS2009}, which is inspired by \cite{MS2003}. 
\begin{proof}[Another proof of Theorem \ref{Main01}]
Starting from $(u_{0},b_{0},M_{0},Q_{0},q_{0})\equiv0$ in $\mathbb{R}_{x}^{n}\times\mathbb{R}_{t}$ and using an inductive argument, we can successively construct a sequence of functions $z_{j}=(u_{j},b_{j},M_{j},Q_{j},q_{j})\in X_{0}$ and select a sequence of small positive constants $r_{j}<2^{-j}$ such that for $m\geq2$ and $j\geq1$,
\begin{align}\label{AQE16}
&\|u_{j}\|_{L^{m}(\Omega)}^{m}+\|b_{j}\|_{L^{m}(\Omega)}^{m}\notag\\
&\geq\|u_{j-1}\|_{L^{m}(\Omega)}^{m}+\|b_{j-1}\|_{L^{m}(\Omega)}^{m}+\frac{\beta}{2}\big(2|\Omega|-\|u_{j-1}\|_{L^{m}(\Omega)}^{m}-\|b_{j-1}\|_{L^{m}(\Omega)}^{m}\big)^{m},
\end{align}	
and
\begin{align}\label{AQE18}
\begin{cases}	
\|z_{j}-z_{j}\ast\rho_{r_{j}}\|_{L^{m}(\Omega)}<2^{-j},\\
\|(z_{j}-z_{j-1})\ast\rho_{r_{i}}\|_{L^{m}(\Omega)}<2^{-j},\quad\forall\,0\leq i\leq j-1,	
\end{cases}
\end{align}	
where the constant $\beta$ is given by \eqref{AQ17} and $\rho_{r}=r^{-(n+1)}\rho(rx,rt)$ is the classical mollifying kernel on $\mathbb{R}^{n+1}$. In fact, assume that the results in \eqref{AQE16}--\eqref{AQE18} hold for any $j\leq k$, then we construct $z_{k+1}=(u_{k+1},b_{k+1},M_{k+1},Q_{k+1},q_{k+1})\in X_{0}$ such that \eqref{AQE16}--\eqref{AQE18} also hold for $j=k+1$. From Lemma \ref{lem10}, we obtain a sequence $\{u_{k,l}\}\subset X_{0}$ satisfying that 
\begin{align*}
	(u_{k,l},b_{k,l},M_{k,l},Q_{k,l},q_{k,l})\overset{\ast}{\xrightharpoonup{\quad}}(u_{k},b_{k},M_{k},Q_{k},q_{k}) \quad\mathrm{in}\;L^{\infty},
\end{align*}	
and
\begin{align*}
	&\liminf_{l\rightarrow\infty}\big(\|u_{k,l}\|_{L^{m}(\Omega)}^{m}+\|b_{k,l}\|_{L^{m}(\Omega)}^{m}\big)\notag\\
	&\geq\|u_{k}\|_{L^{m}(\Omega)}^{m}+\|b_{k}\|_{L^{m}(\Omega)}^{m}+\beta\big(2|\Omega|-\|u_{k}\|_{L^{m}(\Omega)}^{m}-\|b_{k}\|_{L^{m}(\Omega)}^{m}\big)^{m}.
\end{align*}
By the same arguments as in \eqref{AQ09}, we obtain that for any $r>0$, if $l\rightarrow\infty$,
\begin{align*}
	(u_{k,l},b_{k,l},M_{k,l},Q_{k,l},q_{k,l})\ast\rho_{r}\longrightarrow(u_{k},b_{k},M_{k},Q_{k},q_{k})\ast\rho_{r}, \quad\mathrm{in}\;L^{m}.	
\end{align*}	
Then by taking a sufficiently large constant $l_{0}>0$ and a small enough constant $0<r_{k+1}<2^{-(k+1)}$, and setting
\begin{align*}
 (u_{k+1},b_{k+1},M_{k+1},Q_{k+1},q_{k+1}):=(u_{k,l_{0}},b_{k,l_{0}},M_{k,l_{0}},Q_{k,l_{0}},q_{k,l_{0}}),
\end{align*} 
we demonstrate that \eqref{AQE16}--\eqref{AQE18} holds for $j=k+1$. 

Since the sequence $\{z_{k}\}$ remains bounded in $L^{\infty}(\mathbb{R}_{x}^{n}\times\mathbb{R}_{t})$, we obtain that there exists a subsequence converging weakly* to an element $z=(u,b,M,Q,q)\in X$. Without loss of generality, we relabel and denote this subsequence again by $\{z_{k}\}$. In exactly the same way as in \eqref{AQ09}, we deduce that for any given $r>0,$
\begin{align*}
\|(z_{j}-z)\ast\rho_{r}\|_{L^{m}(\Omega)}\longrightarrow0,\quad\text{as $j\rightarrow\infty$},	
\end{align*}
which, together with \eqref{AQE18}, leads to that
\begin{align*}
&\limsup_{k\rightarrow\infty}\|(z_{k}-z)\ast\rho_{r_{k}}\|_{L^{m}(\Omega)}\notag\\
&\leq\limsup_{k\rightarrow\infty}\limsup_{l\rightarrow\infty}\bigg(\sum^{l}_{j=1}\|(z_{k+j-1}-z_{k+j})\ast\rho_{r_{k}}\|_{L^{m}(\Omega)}+\|(z_{k+l}-z)\ast\rho_{r_{k}}\|_{L^{m}(\Omega)}\bigg)\notag\\
&\leq\limsup_{k\rightarrow\infty}\limsup_{l\rightarrow\infty}\sum^{l}_{j=1}\|(z_{k+j-1}-z_{k+j})\ast\rho_{r_{k}}\|_{L^{m}(\Omega)}\leq\limsup_{k\rightarrow\infty}\sum^{\infty}_{j=1}2^{-(k+j)}=0.
\end{align*}	
It then follows from \eqref{AQE18} again that
\begin{align*}
&\limsup_{k\rightarrow\infty}\|z_{k}-z\|_{L^{m}(\Omega)}\notag\\
&\leq\limsup_{k\rightarrow\infty}\big(\|z_{k}-z_{k}\ast\rho_{r_{k}}\|_{L^{m}(\Omega)}+\|(z_{k}-z)\ast\rho_{r_{k}}\|_{L^{m}(\Omega)}+\|z\ast\rho_{r_{k}}-z\|_{L^{m}(\Omega)}\big)\notag\\
&\leq\limsup_{k\rightarrow\infty}2^{-k}=0.
\end{align*}	
This, in combination with \eqref{AQE16}, shows that $|u(x,t)|=|b(x,t)|=1$ for almost all $(x,t)\in\Omega$. Then the proof is finished by appealing to Lemma \ref{lem08}.

\end{proof}

\begin{remark}
It is evident from the preceding proofs that the two approaches, despite their distinct outward formulations, are intrinsically unified by the same underlying mechanism: mollification as the bridge from weak to strong convergence.	For more detailed discussions on these two methods, see e.g. \cite{K2003,S2006}.
\end{remark}	
It remains to prove Theorem \ref{Main02}.
\begin{proof}[Proof of Theorem \ref{Main02}]
For every solution $(u,b,p)$ given by Theorem \ref{Main01}, it follows from the above proofs that there exists a sequence $\{(u_{k},b_{k},M_{k},Q_{k},p_{k})\}\subset X_{0}$ such that for any $m\geq 2$,
\begin{align}\label{E10}
	(u_{k},b_{k},M_{k},Q_{k},p_{k})\longrightarrow(u,b,M,Q,p)\quad\text{in $L^{m}$},	
\end{align}	
where
\begin{align}\label{E12}
	M=u\otimes u-b\otimes b,\quad Q=b\otimes u-u\otimes b.	
\end{align}	
Take
\begin{align*}
	\begin{cases}
		f_{k}=\mathrm{div}(u_{k}\otimes u_{k}-b_{k}\otimes b_{k}-M_{k})=:\mathrm{div}F_{k},\\
		g_{k}=\mathrm{div}(b_{k}\otimes u_{k}-u_{k}\otimes b_{k}-Q_{k})=:\mathrm{div}G_{k}.	
	\end{cases}		
\end{align*}	
Due to the fact that $\max\{\|u_{k}\|_{L^{\infty}},\|b_{k}\|_{L^{\infty}},\|u\|_{L^{\infty}},\|b\|_{L^{\infty}}\}\leq1$, it then follows from \eqref{E10}--\eqref{E12} that
\begin{align*}
	\|F_{k}\|_{L^{m}(\Omega)}\rightarrow0,\quad\|G_{k}\|_{L^{m}(\Omega)}\rightarrow0,\quad\mathrm{as}\;k\rightarrow\infty.
\end{align*}
Denote $m'=\frac{m}{m-1}$. For any $\varphi\in W_{0}^{1,m'}(\Omega)$ with $\|\varphi\|_{W^{1,m'}}\leq1$, we have
\begin{align*}
	|\langle f_{k},\varphi \rangle|=\bigg|\int_{\Omega}F_{k}:\nabla\varphi dxdt\bigg|\leq\|F_{k}\|_{L^{m}(\Omega)}\|\nabla\varphi\|_{L^{m'}(\Omega)}\rightarrow0,	\quad\mathrm{as}\;k\rightarrow\infty.
\end{align*}	
Similarly, we also have $|\langle g_{k},\varphi \rangle|\rightarrow0$. The proof is complete.

\end{proof}


%

\noindent{\bf{\large Acknowledgements.}} This work was partially supported by the National Key research and development program of
China (No.\ 2022YFA1005700). C. Miao was partially supported by the National Natural Science Foundation of China (No. 12371095). Z. Zhao was partially supported by NSFC (No.\ 12501254).




\end{document}